\documentclass[11pt,a4paper]{amsart}
\usepackage[margin=2.6cm]{geometry}
\usepackage{amssymb,booktabs,enumitem,longtable}
\usepackage[hidelinks]{hyperref}
\newtheorem{theorem}{Theorem}[section]
\newtheorem{proposition}[theorem]{Proposition}
\newtheorem{lemma}[theorem]{Lemma}
\newtheorem{corollary}[theorem]{Corollary}

\newtheorem{conjecture}[theorem]{Conjecture}
\theoremstyle{definition}
\newtheorem{definition}[theorem]{Definition}
\theoremstyle{remark}
\newtheorem{remark}[theorem]{Remark}

\newcommand{\s}{\mathsf s}
\newcommand{\F}{\mathbb F}
\newcommand{\Z}{\mathbb Z}
\newcommand{\supp}{\operatorname{supp}}
\newcommand{\conv}{\operatorname{conv}}

\title[The EGZ constant of $C_n^r$ for even $n$]{The Erd\H{o}s--Ginzburg--Ziv constant of $C_n^r$ for even $n$:
rigid sequences, products, and growth beyond Harborth's bound}
\author{Christian Elsholtz}
\address{Institut f\"ur Analysis und Zahlentheorie, Technische Universit\"at Graz, Kopernikusgasse 24,
8010 Graz, Austria}
\email{elsholtz@math.tugraz.at}
\subjclass[2020]{11B30, 11P70, 20K01, 05B25}
\keywords{Erd\H{o}s--Ginzburg--Ziv constant, zero-sum sequences, Property D, Property C, product constructions}

\begin{document}

\begin{abstract}
Let $\s(C_n^r)$ be the Erd\H{o}s--Ginzburg--Ziv constant of $C_n^r$ and $\eta(C_n^r)$ the short zero-sum constant.
For even $n$, Harborth's construction gives $\s(C_n^r)\ge2^r(n-1)+1$, with equality when $n$ is a power of two.
Wang and Zeng recently showed that this bound is strict for every $r\ge4$ and every even $n$ that is not a power of
two, with an excess of $2^{r-3}(m-1)$, where $m$ is the odd part of $n$; this disproves Properties C and D. We show that
the excess grows faster than $2^r$. For every even $n=2^am$ with odd $m\ge3$ and every $r\ge4$,
\[
\s(C_n^r)\ge2^r(n-1)+(m-1)\,2^{r-3}\lfloor r/4\rfloor+1 ,
\]
so that $\s(C_n^r)-2^r(n-1)\ge c_n\,r\,2^r$. The proof introduces \emph{rigid} sequences, whose zero-sums on their
support have a very restricted form, and a product theorem for them; the sequences of Wang and Zeng are rigid. For
$n\equiv2\pmod4$ we prove $\s(C_n^5)\ge32(n-1)+8(m-1)+1$ (improving on Wang and Zeng, whose excess in rank five is
$4(m-1)$), and also obtain $\s(C_n^r)\ge2^r(n-1)+(m-1)2^{r-2}\lfloor r/5\rfloor+1$ for all $r\ge5$. This comes from a new
rank-five block, valid for all such $n$ at once; its proof is a case analysis with explicit certificates that do not
depend on $n$. Corresponding bounds hold for $\eta$. We also record a divisor lemma and a reformulation of Property D
at a critical length, and we show that the $n$-term subsequence sums of a maximal sequence with a term of multiplicity
$n-1$ cover all nonzero elements of the group.
\end{abstract}

\maketitle

\section{Introduction}

For a finite abelian group $G$ let $\s(G)$ be the smallest $\ell$ such that every sequence of $\ell$ elements of
$G$ has a zero-sum subsequence of length \emph{exactly} $\exp(G)$, and let $\eta(G)$ be the smallest $\ell$ such
that every such sequence has a \emph{short} zero-sum subsequence, of length in $[1,\exp(G)]$. The theorem of
Erd\H{o}s, Ginzburg and Ziv gives $\s(C_n)=2n-1$, and Reiher \cite{Reiher} proved $\s(C_n^2)=4n-3$. For even $n$,
Harborth's construction \cite{Harborth} (every vector of $\{0,1\}^r$ with multiplicity $n-1$) gives
\begin{equation}\label{eq:harborth}
\s(C_n^r)\ge 2^r(n-1)+1,\qquad \eta(C_n^r)\ge(2^r-1)(n-1)+1 ,
\end{equation}
and both inequalities are equalities when $n$ is a power of two. For odd $n$ the best general lower bounds in ranks
three and four are $\s(C_n^3)\ge9n-8$ and $\s(C_n^4)\ge20n-19$ \cite{Els04,EEGKR}, and in ranks five and six
$\s(C_n^5)\ge42n-41$, $\s(C_n^6)\ge96n-95$ \cite{Edel04b}.

\begin{definition}[{Gao \cite{Gao00}; see {\cite[Section~7]{GG}}}]\label{def:CD}
Let $G=C_n^r$. $G$ has \emph{Property C} if every sequence over $G$ of length $\eta(G)-1$ without short zero-sum
subsequence has the form $T^{n-1}$, and \emph{Property D} if every sequence over $G$ of length $\s(G)-1$ without
zero-sum subsequence of length $n$ has the form $T^{n-1}$.
\end{definition}

Property D was introduced by Gao \cite{Gao00}; Property C goes back to van Emde Boas \cite{vEB} in rank two. It was
conjectured that both hold for all $C_n^r$ \cite{GG,GHST}. For $r=2$ and $n\in\{2,3,5,7\}$ Property D is known
\cite{ST}; see also \cite{GG}. For these $n$ the value $\s(C_n^2)=4n-3$ is due to Kemnitz \cite{Kemnitz}. Property D
is multiplicative in rank two \cite{Gao00}, and Schlage-Puchta \cite{SP} proved it for all sufficiently large
primes. In rank three, Gao and Thangadurai \cite{GT} proved Property D for $C_{2^a}^3$ and $C_{3^a}^3$, and Gao, Hou,
Schmid and Thangadurai {\cite[Proposition~5.6, Theorem~1.9]{GHST}} for $C_{5^a}^3$.

\subsection*{The parity obstruction and the result of Wang and Zeng}
For even $n$ both properties force equality in \eqref{eq:harborth}. The reason is parity: if $T^{n-1}$ has no
zero-sum subsequence of length $n$, then $T$ injects into $C_2^r$ (Lemma~\ref{lem:div} with $d=2$), so $|T|\le2^r$.
Consequently a single sequence of length greater than $2^r(n-1)$ without zero-sum subsequence of length $n$
disproves Property D, whatever the true value of $\s(C_n^r)$ (Corollary~\ref{cor:D}). The first such sequence, of
length $81$ in $C_6^4$, was found by Wang \cite{Wang1}. Wang and Zeng \cite{WZ} then proved:

\begin{theorem}[Wang and Zeng \cite{WZ}]\label{thm:WZ}
Let $r\ge4$ and let $n\ge2$ be even. Then $C_n^r$ has Property D, and also Property C, if and only if $n$ is a power
of two. More precisely, if $n=2^am$ with $a\ge1$ and odd $m\ge3$, then
\[
\s(C_n^r)\ge2^r(n-1)+2^{r-3}(m-1)+1,\qquad \eta(C_n^r)\ge(2^r-1)(n-1)+2^{r-3}(m-1)+1 .
\]
\end{theorem}

\subsection*{A note on the versions}
The first version of the present paper was written before we saw version~2 of \cite{WZ}; at that time only the
case $n=6$ \cite{Wang1} was available. Our version~1 proved the failure of Properties C and D for all
$n\equiv2\pmod4$, $n\ge6$, $r\ge4$, with excess $2^{r-2}$, by an $18$-point construction of almost constant vectors,
based on our earlier methods in \cite{EEGKR,Els04,Els16}. This result is now contained in Theorem~\ref{thm:WZ}. In the
present version we improve the bounds of Wang and Zeng in all ranks $r\ge8$, and for $n\equiv2\pmod4$ in all ranks
$r\ge5$.

\subsection*{Main results}
In Theorem~\ref{thm:WZ} the excess over Harborth's bound is a fixed multiple of $2^r$; all higher ranks are
obtained by doubling. We show that the excess grows faster than $2^r$.

\begin{theorem}\label{thm:main}
Let $n=2^am$ with $a\ge1$ and odd $m\ge3$, and let $r\ge4$. Then
\begin{align*}
\s(C_n^r)&\ge2^r(n-1)+(m-1)\,2^{r-3}\Big\lfloor\frac r4\Big\rfloor+1,\\
\eta(C_n^r)&\ge(2^r-1)(n-1)+(m-1)\,2^{r-3}\Big\lfloor\frac r4\Big\rfloor+1 .
\end{align*}
\end{theorem}

For $4\le r\le7$ this is Theorem~\ref{thm:WZ}; for $r\ge8$ it is stronger. For $n\equiv2\pmod4$ we have a better
block in rank five.

\begin{theorem}\label{thm:main2}
Let $n=2m$ with odd $m\ge3$.
\begin{enumerate}[label=\textup{(\alph*)}]
\item $\s(C_n^5)\ge32(n-1)+8(m-1)+1$ and $\eta(C_n^5)\ge31(n-1)+8(m-1)+1$.
\item For all $r\ge5$,
\[
\s(C_n^r)\ge2^r(n-1)+(m-1)\,2^{r-2}\Big\lfloor\frac r5\Big\rfloor+1,\qquad
\eta(C_n^r)\ge(2^r-1)(n-1)+(m-1)\,2^{r-2}\Big\lfloor\frac r5\Big\rfloor+1 .
\]
\item Using also the rank-four block of Wang and Zeng: for $r\ge4$ let $\mu(r)=\max\{a/4+b/8:\ a,b\in\Z_{\ge0},\ 5a+4b\le r\}$.
Then $\s(C_n^r)\ge2^r(n-1)+(m-1)\,2^r\mu(r)+1$ and $\eta(C_n^r)\ge(2^r-1)(n-1)+(m-1)\,2^r\mu(r)+1$.
\end{enumerate}
\end{theorem}

Parts (a) and (b) use only the rank-five block of Section~\ref{sec:rank5} and the product theorem for glued sequences
(Theorem~\ref{thm:gluedprod}); they do not depend on \cite{WZ}. Since $2\lfloor r/5\rfloor\ge\lfloor r/4\rfloor$, part (b)
is at least as strong as Theorem~\ref{thm:main} for $n\equiv2\pmod4$, and strictly stronger for $r\in\{5,6,7\}$ and
$r\ge10$. Part (c) improves (b) slightly for some $r$, for instance $\mu(9)=3/8$ against $\lfloor9/5\rfloor/4=1/4$.

\begin{table}[h]
\centering\small
\begin{tabular}{@{}ccc@{}}
\toprule
$r$ & excess, Theorem~\ref{thm:WZ} & excess, Theorems~\ref{thm:main} and \ref{thm:main2}\\
\midrule
$4$ & $2(m-1)$ & $2(m-1)$\\
$5$ & $4(m-1)$ & $8(m-1)$ for $n\equiv2\ (4)$\\
$6$ & $8(m-1)$ & $16(m-1)$ for $n\equiv2\ (4)$\\
$7$ & $16(m-1)$ & $32(m-1)$ for $n\equiv2\ (4)$\\
$8$ & $32(m-1)$ & $64(m-1)$\\
$9$ & $64(m-1)$ & $128(m-1)$; $192(m-1)$ for $n\equiv2\ (4)$\\
$10$ & $128(m-1)$ & $256(m-1)$; $512(m-1)$ for $n\equiv2\ (4)$\\
$14$ & $2048(m-1)$ & $6144(m-1)$; $10240(m-1)$ for $n\equiv2\ (4)$\\
$15$ & $4096(m-1)$ & $12288(m-1)$; $24576(m-1)$ for $n\equiv2\ (4)$\\
$19$ & $65536(m-1)$ & $262144(m-1)$; $458752(m-1)$ for $n\equiv2\ (4)$\\
$20$ & $131072(m-1)$ & $655360(m-1)$; $1048576(m-1)$ for $n\equiv2\ (4)$\\
\bottomrule
\end{tabular}
\caption{The excess $\s(C_n^r)-1-2^r(n-1)$ guaranteed by the constructions, $n=2^am$, $m$ odd.}
\end{table}

\begin{corollary}\label{cor:growth}
Let $n\ge2$ be even, and let $m$ be the odd part of $n$.
\begin{enumerate}[label=\textup{(\alph*)}]
\item \textup{(Harborth \cite{Harborth})} If $n$ is a power of two, then $\s(C_n^r)=2^r(n-1)+1$ for all $r$.
\item Otherwise, for all $r\ge4$,
\[
\s(C_n^r)-2^r(n-1)-1\ \ge\ (m-1)\,2^{r-3}\Big\lfloor\frac r4\Big\rfloor\ \ge\ \frac{m-1}{64}\,r\,2^r .
\]
\item If $n\equiv2\pmod4$, then for all $r\ge5$,
\[
\s(C_n^r)-2^r(n-1)-1\ \ge\ (m-1)\,2^{r-2}\Big\lfloor\frac r5\Big\rfloor\ \ge\ \frac{m-1}{40}\,r\,2^r .
\]
\end{enumerate}
In particular $\big(\s(C_n^r)-1\big)/2^r-(n-1)\to\infty$ unless $n$ is a power of two. The same bounds hold for
$\eta(C_n^r)-(2^r-1)(n-1)-1$.
\end{corollary}

\begin{proof}
(a) is Harborth's theorem \cite{Harborth}. (b) follows from Theorem~\ref{thm:main} and $\lfloor r/4\rfloor\ge r/8$ for
$r\ge4$, and (c) from Theorem~\ref{thm:main2}(b) and $\lfloor r/5\rfloor\ge r/10$ for $r\ge5$.
\end{proof}

The proof of Theorem~\ref{thm:main} rests on a simple notion. A zero-sum free sequence with terms of multiplicity
$n-1$ (heavy) and at most $m-1$ (light) is \emph{rigid} if every $n$-term zero-sum supported on its support, with
arbitrary multiplicities, is of the form $x^n$ or $h^a\ell^{n-a}$ with $h$ heavy and $\ell$ light
(Definition~\ref{def:rigid}). Products of rigid sequences are zero-sum free (Theorem~\ref{thm:rigidprod}), and the
sequences of Wang and Zeng are rigid (Proposition~\ref{prop:WZrigid}). Light densities add under products, so the
excess per $2^r$ grows linearly in the number of factors. For $n\equiv2\pmod4$ the rank-five block of
Theorem~\ref{thm:main2}(a) is built from two copies of the $20$-point cone in $\{-1,0,1\}^4$; its proof is a case
analysis whose certificates do not depend on $n$ (Section~\ref{sec:rank5}).

We also record the general tools: a divisor lemma (Lemma~\ref{lem:div}), which gives
$M_m(3)=9$ and $M_m(4)=20$ for odd $m$ with $3\mid m$, and a reformulation of Property~D at a critical length
(Proposition~\ref{prop:equiv}).

\subsection*{Organisation}
Section~\ref{sec:prelim} collects notation and basic facts, and Section~\ref{sec:div} the divisor lemma and the
critical length. Section~\ref{sec:rigid} contains the product theorem for rigid sequences, and
Section~\ref{sec:WZ} shows that the sequences of Wang and Zeng are rigid and proves Theorem~\ref{thm:main}.
Section~\ref{sec:glued} treats glued sequences for $n=2m$, and Section~\ref{sec:rank5} the rank-five block and
Theorem~\ref{thm:main2}. Section~\ref{sec:last} shows that the $n$-term sums of a maximal sequence with a term of multiplicity $n-1$ cover all nonzero elements, and Section~\ref{sec:deps} explains the role of computation. The appendix lists the certificates for the rank-five block.

\section{Notation and preliminaries}\label{sec:prelim}

We write sequences multiplicatively, $S=\prod_{g}g^{v_g(S)}$, and $T^k$ for the sequence in which every term of $T$ is
repeated $k$ times. The support $\supp(S)$ is the set of distinct terms. Elements of $C_n^r=(\Z/n\Z)^r$ are identified
with integer vectors, and $\pi_d$ is reduction modulo $d\mid n$. A set $T\subseteq C_m^r$ is \emph{$m$-good} if
$T^{m-1}$ has no zero-sum subsequence of length $m$. For even $n=2^am$ with $m$ odd, a term of a sequence over
$C_n^r$ is \emph{heavy} if its multiplicity is $n-1$, and \emph{light} if its multiplicity is at most $m-1$. Let $M_m(r)$ (resp.\ $N_m(r)$) be the maximal size of a set
$T\subseteq C_m^r$ such that $T^{m-1}$ has no zero-sum subsequence of length $m$ (resp.\ no short zero-sum
subsequence). A cap in $AG(r,3)=C_3^r$ is a set without three collinear points; $a_3(r)$ is the maximal size of a
cap, with $a_3(r)=2,4,9,20,45,112$ for $r\le6$ \cite{EFLS,Potechin}. For larger $r$ see Edel \cite{Edel04}.

\begin{lemma}\label{lem:basic}
Let $G=C_m^r$. (1) If $T^{m-1}$ has no zero-sum subsequence of length $m$, then $T$ has no repeated terms.
(2) $\s(G)\ge(m-1)M_m(r)+1$ and $\eta(G)\ge(m-1)N_m(r)+1$. (3) If $G$ has Property D, then $\s(G)-1=(m-1)M_m(r)$; if
$G$ has Property C, then $\eta(G)-1=(m-1)N_m(r)$. (4) Translations preserve the absence of zero-sum subsequences of
length $m$, and $N_m(r)\ge M_m(r)-1$.
\end{lemma}

\begin{proof}
(1) A repeated term would occur $2(m-1)\ge m$ times. (2) is immediate. (3) By Property~D an extremal sequence is
$T^{m-1}$ with a set $T$, so $\s(G)-1=(m-1)|T|\le(m-1)M_m(r)$, and (2) gives equality; similarly for $\eta$. (4) A
translate changes the sum of $m$ terms by $mg=0$. If $0\in T$ and $T$ is $m$-good, a short zero-sum of
$(T\setminus\{0\})^{m-1}$ of length $k$ is completed by $0^{m-k}$ to one of length $m$.
\end{proof}

\begin{lemma}[{\cite[Lemma~3.2]{Wang1}, \cite[Lemma~3.4]{WZ}}]\label{lem:eta}
If $S$ over $G$, $\exp(G)=n$, has no $n$-term zero-sum subsequence and contains a term of multiplicity $n-1$, then
$\eta(G)\ge|S|-n+2$.
\end{lemma}

\begin{lemma}[doubling, see \cite{Els04,WZ}]\label{lem:double}
If $S$ over $C_n^r$ has no $n$-term zero-sum subsequence, then $S\times\{0,1\}$ (every term $g$ replaced by $(g,0)$
and $(g,1)$) is a sequence over $C_n^{r+1}$ of length $2|S|$ without $n$-term zero-sum subsequence, and with the
same multiplicities.
\end{lemma}

\begin{proof}
$n$ entries from $\{0,1\}$ sum to $0$ modulo $n$ only if they are all equal.
\end{proof}

\section{The divisor lemma and the critical length}\label{sec:div}

\begin{lemma}\label{lem:div}
Let $d\mid m$ with $d\ge2$, and let $T\subseteq C_m^r$ be a set. If $T^{m-1}$ has no zero-sum subsequence of length
$m$ (resp.\ no short zero-sum subsequence), then $\pi_d$ is injective on $T$ and $\pi_d(T)^{d-1}$ has no zero-sum
subsequence of length $d$ (resp.\ no short zero-sum subsequence). In particular $M_m(r)\le M_d(r)$ and
$N_m(r)\le N_d(r)$.
\end{lemma}

\begin{proof}
Multiplying a relation modulo $d$ by $m/d$ gives a relation modulo $m$, while all multiplicities stay at most
$m-1$: for $0\le b_t\le d-1$ let $W=\prod_tt^{(m/d)b_t}$, a subsequence of $T^{m-1}$ of length $(m/d)\sum b_t$; if
$\sum b_tt\equiv0\pmod d$, then $\sum(m/d)b_tt\equiv0\pmod m$. If $\pi_d(t)=\pi_d(t')$ with $t\ne t'$, take
$b_t=d-1$, $b_{t'}=1$; then $W$ is a zero-sum of length $m$. A zero-sum subsequence of $\pi_d(T)^{d-1}$ of length $d$
(resp.\ in $[1,d]$) gives one of $T^{m-1}$ of length $m$ (resp.\ in $[1,m]$).
\end{proof}

For $d=2$ this is {\cite[Proposition~3.1]{Wang1}}; lifting arguments of this kind are used in \cite{EEGKR}.

\begin{corollary}\label{cor:values}
(1) If $m$ is even, then $M_m(r)=2^r$ and $N_m(r)=2^r-1$. (2) If $3\mid m$, then $M_m(r)\le a_3(r)$ and
$N_m(r)\le a_3(r)-1$. (3) If $m$ is odd and $3\mid m$, then $M_m(3)=9$, $M_m(4)=20$, $N_m(3)=8$, $N_m(4)=19$.
\end{corollary}

\begin{proof}
$\{0,1\}^r$ is $m$-good for every $m$, since $m$ binary vectors sum to $0$ modulo $m$ only if they are all equal;
every subset of $C_2^r$ is $2$-good, and $T\subseteq C_2^r$ has no short zero-sum iff $0\notin T$. For $d=3$,
$t+t+t'\ne0$ for $t\ne t'$, and $t+t'+t''=0$ for distinct terms iff they are collinear, so $M_3(r)=a_3(r)$ and
$N_3(r)=a_3(r)-1$. Lemma~\ref{lem:div} gives (1) and (2). For (3) use $a_3(3)=9$, $a_3(4)=20$, the sets of size $9$
and $20$ that are $m$-good for every odd $m$ \cite{Els04,EEGKR,Els16}, and Lemma~\ref{lem:basic}(4).
\end{proof}

\begin{proposition}\label{prop:equiv}
Let $G=C_m^r$ and $L=(m-1)M_m(r)$. Then $G$ has Property~D if and only if $\s(G)=L+1$ and every sequence of length
$L$ without zero-sum subsequence of length $m$ has the form $T^{m-1}$. In particular, if there is a sequence of
length at least $L$ without zero-sum subsequence of length $m$ which is not of the form $T^{m-1}$, then $G$ does not
have Property~D, whatever the value of $\s(G)$. The analogous statements hold for Property~C, with $N_m(r)$.
\end{proposition}

\begin{proof}
Lemma~\ref{lem:basic}(3) gives one direction. Conversely the two conditions are Property~D. If $S$ is such a sequence
and $\s(G)-1>L$, the first condition fails; otherwise $\s(G)-1\le L\le|S|\le\s(G)-1$ and the second fails.
\end{proof}

\begin{corollary}\label{cor:D}
Let $m$ be even. If $C_m^r$ has Property~D, then $\s(C_m^r)=2^r(m-1)+1$; if it has Property~C, then
$\eta(C_m^r)=(2^r-1)(m-1)+1$. A single sequence of length greater than $2^r(m-1)$ without zero-sum subsequence of
length $m$ disproves Property~D, and one of length greater than $(2^r-1)(m-1)$ without short zero-sum subsequence
disproves Property~C.
\end{corollary}

\begin{proof}
Lemma~\ref{lem:basic}(3), Corollary~\ref{cor:values}(1) and \eqref{eq:harborth}.
\end{proof}

For odd $m$ with $3\mid m$, Corollary~\ref{cor:values}(3) shows that in ranks $3$ and $4$ the critical lengths are
$9(m-1)$ and $20(m-1)$. Whenever $\s(C_p^r)=c_r(p-1)+1$ for every prime $p\mid m$ (with $c_3=9$, $c_4=20$), this is the
extremal length by {\cite[Corollary~4.5]{EEGKR}}; for example $\s(C_{15}^3)=127$, since $\s(C_3^3)=19$ \cite{EEGKR} and
$\s(C_5^3)=37$ \cite{GHST}. Property~D then becomes a uniqueness statement at the critical length.

\section{Rigid sequences and their products}\label{sec:rigid}

Throughout, $n=2^am$ with $a\ge1$ and odd $m\ge3$.

\begin{definition}\label{def:rigid}
A sequence $S$ over $C_n^r$ without zero-sum subsequence of length $n$ is \emph{rigid} if every term has
multiplicity $n-1$ (\emph{heavy}) or at most $m-1$ (\emph{light}), and every $n$-term zero-sum $W$ with
$\supp(W)\subseteq\supp(S)$, with arbitrary multiplicities, is of the form $W=x^n$, or $W=h^a\ell^{n-a}$ with $h$
heavy and $\ell$ light.
\end{definition}

\begin{definition}\label{def:product}
Let $S_i$ be a sequence over $C_n^{r_i}$ with heavy and light terms, $i=1,\dots,t$. The \emph{product}
$S_1\otimes\dots\otimes S_t$ is the sequence over $C_n^{r_1+\dots+r_t}$ whose terms are the tuples
$(x_1,\dots,x_t)$, $x_i\in\supp(S_i)$, with at most one light coordinate; a tuple has multiplicity $n-1$ if all
coordinates are heavy, and $m-1$ otherwise.
\end{definition}

If $S_i$ has $H_i$ heavy and $L_i$ light terms, the product has $\prod H_i$ heavy and
$\sum_iL_i\prod_{j\ne i}H_j$ light terms. Writing $L_i=\lambda_iH_i$, the light density of the product is
$\lambda_1+\dots+\lambda_t$, and its length is
\begin{equation}\label{eq:prodlength}
\Big(\prod_iH_i\Big)\Big((n-1)+(m-1)\sum_i\lambda_i\Big).
\end{equation}

\begin{theorem}\label{thm:rigidprod}
If $S_1,\dots,S_t$ are rigid, then $S_1\otimes\dots\otimes S_t$ has no zero-sum subsequence of length $n$.
\end{theorem}

\begin{proof}
Induction on $t$; the case $t=1$ is the hypothesis. Let $P=S_1\otimes\dots\otimes S_t$ and let $W$ be an $n$-term
zero-sum subsequence of $P$. Its projection $W^{(j)}$ to the $j$-th block of coordinates is an $n$-term zero-sum
supported on $\supp(S_j)$, so by rigidity it is $x^n$ or $h_j^{a_j}\ell_j^{n-a_j}$.

Suppose first that $W^{(j)}=x^n$ for some $j$. Then every term of $W$ has $j$-th coordinate $x$. Deleting this
coordinate gives an $n$-term zero-sum $W'$ supported on $\supp(P')$, $P'=\bigotimes_{i\ne j}S_i$. If $x$ is heavy, a
term $(y,x)$ has the same multiplicity in $P$ as $y$ in $P'$. If $x$ is light, all other coordinates of $(y,x)$ are
heavy, and its multiplicity $m-1$ in $P$ is at most the multiplicity $n-1$ of $y$ in $P'$. So $W'$ is a subsequence of
$P'$, which contradicts the induction hypothesis.

Otherwise $W^{(j)}=h_j^{a_j}\ell_j^{n-a_j}$ for every $j$. Here $n-a_j\ge m$: otherwise $a_j\le n-1$ and
$n-a_j\le m-1$, and $W^{(j)}$ would be a subsequence of $S_j$. Consider the $n-a_j$ terms of $W$ with $j$-th
coordinate $\ell_j$. Each has all other coordinates heavy, and its $i$-th coordinate lies in
$\supp(W^{(i)})=\{h_i,\ell_i\}$, so it equals $h_i$. Hence all these terms are the single element
$(h_1,\dots,\ell_j,\dots,h_t)$, whose multiplicity in $P$ is $m-1<n-a_j$. This contradicts $W\mid P$.
\end{proof}

\begin{remark}
The theorem uses neither the decomposition $C_n=C_{2^a}\oplus C_m$ nor any labels; it only uses rigidity. A product
of rigid sequences need not be rigid: if $W^{(1)}=h^a\ell^{n-a}$ and $W^{(2)}=h'^b\ell'^{n-b}$ with $a+b\ge n$, then
$(h,h')^{a+b-n}(h,\ell')^{n-b}(\ell,h')^{n-a}$ is a zero-sum on the support of $S_1\otimes S_2$. This is why we take
products of rigid factors directly rather than iterating.
\end{remark}

\section{The sequences of Wang and Zeng are rigid}\label{sec:WZ}

We recall the construction of {\cite[Section~3]{WZ}}. Write $n=km$ with $k=2^a\ge2$ and odd $m\ge3$. The nine labels
$P=(x,y,z)$ of Table~\ref{tab:WZ} form the set of Elsholtz {\cite[p.~354]{Els04}}. For each label and each listed
value $t$, let $q(P,t)\in\{0,\dots,n-1\}$ be determined by $q\equiv x\pmod m$, $q\equiv b(P,t)\pmod k$, and put
\[
g(P,t)=\big(y,\ x-z,\ q(P,t),\ t\big)\in C_n^4 .
\]
The sequence $S_{WZ}(n)$ consists of the $16$ heavy terms $g(P,t)$ of multiplicity $n-1$ and the $2$ light terms of
multiplicity $m-1$ listed in the table. Its length is $16(n-1)+2(m-1)$, and it has no zero-sum subsequence of length
$n$ {\cite[Proposition~3.5]{WZ}}.

\begin{table}[h]
\centering\small
\begin{tabular}{@{}ccccc@{}}
\toprule
row & $P=(x,y,z)$ & $b(P,t)$ & $t$, heavy & $t$, light\\
\midrule
A & $(0,0,0)$ & $-t$ & $0,1$ & --\\
B & $(1,0,0)$ & $-1$ & $2$ & $2+k$\\
C & $(0,1,0)$ & $-1$ & $0,1$ & --\\
D & $(0,0,1)$ & $-1$ & $1$ & $1+k$\\
E & $(1,0,1)$ & $1-t$ & $0,1$ & --\\
F & $(0,1,1)$ & $-1$ & $0,1$ & --\\
G & $(1,1,2)$ & $0$ & $0,1$ & --\\
H & $(1,2,2)$ & $0$ & $0,1$ & --\\
I & $(2,1,2)$ & $0$ & $0,1$ & --\\
\bottomrule
\end{tabular}
\caption{The construction of Wang and Zeng \cite{WZ}.}\label{tab:WZ}
\end{table}

\begin{proposition}\label{prop:WZrigid}
For every $n=2^am$ with $a\ge1$ and odd $m\ge3$, the sequence $S_{WZ}(n)$ is rigid. More precisely, every $n$-term
zero-sum supported on $\supp(S_{WZ}(n))$ is either $x^n$, or $g(B,2)^{n-j}g(B,2+k)^{j}$ or $g(D,1)^{n-j}g(D,1+k)^{j}$
with $m\mid j$.
\end{proposition}

\begin{proof}
Let $W$ be an $n$-term zero-sum supported on $\supp(S_{WZ}(n))$, with arbitrary multiplicities; we only use that
$W$ has $n$ terms and that its four coordinate sums vanish modulo $n$. The argument follows Steps~1 and 2 of the proof of
{\cite[Proposition~3.5]{WZ}}; we give it in full since no bound on the multiplicities is used.

\emph{Step 1: the triangle.} The first coordinate $y$ takes values in $\{0,1,2\}$, so $\sum y\in\{0,n,2n\}$; the second
coordinate $x-z$ takes values in $\{-1,0,1\}$, so $\sum(x-z)\in\{-n,0,n\}$. Hence the average
$(\bar Y,\bar V)=\big(\sum y,\sum(x-z)\big)/n$ is an integer point, and it lies in the convex hull of the nine label
points $(y,x-z)$, which is the triangle $Y\ge0$, $V\ge-1$, $Y+V\le1$ with vertices $(0,1)$, $(0,-1)$, $(2,-1)$. Its integer
points are the three vertices and $(0,0)$, $(1,0)$, $(1,-1)$, all on the boundary. So one of the three inequalities is an
equality for the average, hence for every term of $W$, and $W$ is supported on one side:
\[
V=-1:\ \text{rows D, F, G, H};\qquad Y=0:\ \text{rows A, B, D, E};\qquad Y+V=1:\ \text{rows B, C, H, I}.
\]

\emph{Step 2: one row.} The third coordinate gives $\sum x\equiv0\pmod m$ and $\sum b\equiv0\pmod k$, and the fourth gives
$\sum t\equiv0\pmod n$. Write $\alpha,\beta,\gamma,\delta,\varepsilon,\zeta,\eta,\theta,\iota$ for the numbers of terms of $W$ in rows
A, B, C, D, E, F, G, H, I.

\emph{Side $V=-1$.} Here $x=0$, $b=-1$ in rows D, F and $x=1$, $b=0$ in rows G, H. With $a=\eta+\theta$ we get $a\equiv0\pmod m$
and $-(n-a)\equiv0\pmod k$, so $n\mid a$ and $a\in\{0,n\}$. If $a=0$, the values $y\in\{0,1\}$ of rows D, F and $\sum
y\in\{0,n\}$ force a single row; if $a=n$, the values $y\in\{1,2\}$ of rows G, H and $\sum y\in\{n,2n\}$ do the same.

\emph{Side $Y=0$.} Here $x-z$ is $0,1,-1,0$ in rows A, B, D, E, so $\beta-\delta\in\{-n,0,n\}$: either $W$ lies in row B or
in row D, or $\beta=\delta$. In the last case $\sum x=\beta+\varepsilon\equiv0\pmod m$, and
\[
\sum b=-\sum_{\mathrm A}t-\beta-\delta+\sum_{\mathrm E}(1-t),\qquad
\sum t=\sum_{\mathrm A}t+(2\beta+kj_{\mathrm B})+(\delta+kj_{\mathrm D})+\sum_{\mathrm E}t ,
\]
where $j_{\mathrm B}$, $j_{\mathrm D}$ count the terms with $t=2+k$ and $t=1+k$. Adding the two congruences modulo $k$
gives $\beta+\varepsilon\equiv0\pmod k$, hence $\beta+\varepsilon\in\{0,n\}$. If $\beta+\varepsilon=n$, then
$\alpha=\delta=0$, so $\beta=\delta=0$ and $W$ lies in row E; if $\beta+\varepsilon=0$, then $W$ lies in row A.

\emph{Side $Y+V=1$.} Here $(y,x-z)$ is $(0,1),(1,0),(2,-1),(1,0)$ in rows B, C, H, I, so $\sum(x-z)=\beta-\theta\in\{-n,0,n\}$:
either $W$ lies in row B or in row H, or $\beta=\theta$ and then $\gamma+\iota=n-2\beta$. In the last case $x$ is $1,0,1,2$
and $b$ is $-1,-1,0,0$ in rows B, C, H, I, so $\beta+\gamma\equiv0\pmod k$ and $2\beta+2\iota\equiv0\pmod m$. As
$\beta+\gamma=n-(\beta+\iota)$ and $m$ is odd, $\beta+\iota\equiv0$ modulo $k$ and modulo $m$, so $\beta+\iota\in\{0,n\}$.
If $\beta+\iota=0$, then $\beta=\theta=0$ and $W$ lies in row C; if $\beta+\iota=n$, then $\gamma=\theta=0$, so $\beta=0$ and $W$
lies in row I.

\emph{Step 3: within a row.} If $W$ lies in a row other than B, D, the fourth coordinates are $0$ and $1$, and $\sum
t\in\{0,n\}$ forces all of them to be equal, so $W=x^n$. In rows B and D the fourth coordinates are $t_0$ and $t_0+k$; if
$j$ terms use $t_0+k$, then $n\mid nt_0+kj$, so $m\mid j$.
\end{proof}

We checked the statement independently by computer for $n=6,10,12,20,24$: on the support of $S_{WZ}(n)$ the
$n$-term zero-sums are exactly the $18$ sequences $x^n$ and the sequences listed in the proposition. For the same $n$
we also verified Theorem~\ref{thm:rigidprod} for $S_{WZ}(n)\otimes S_{WZ}(n)$ directly. These checks are not used in the proofs.

\begin{proof}[Proof of Theorem~\ref{thm:main}]
Let $j=\lfloor r/4\rfloor\ge1$. By Proposition~\ref{prop:WZrigid} and Theorem~\ref{thm:rigidprod}, the $j$-fold
product of $S_{WZ}(n)$ is a sequence over $C_n^{4j}$ without $n$-term zero-sum. Its light density is $j\cdot2/16=j/8$,
so by \eqref{eq:prodlength} its length is $2^{4j}\big((n-1)+(m-1)j/8\big)$. Doubling $r-4j$ times
(Lemma~\ref{lem:double}) gives length $2^r(n-1)+(m-1)2^{r-3}j$ over $C_n^r$. The sequence contains a term of
multiplicity $n-1$, so Lemma~\ref{lem:eta} gives the bound for $\eta$.
\end{proof}

\section{Glued sequences for \texorpdfstring{$n=2m$}{n=2m}}\label{sec:glued}

Theorem~\ref{thm:rigidprod} applies to all $n=2^am$, but its hypothesis, rigidity, is a statement about all zero-sums
on the support, which we can verify for the sequences of Wang and Zeng but not for the rank-five block of
Section~\ref{sec:rank5}. For $n=2m$ there is a second, more concrete route. A \emph{glued} sequence is described by
its labels in $\F_2^r$ and its support modulo $m$; the hypotheses of the corresponding product theorem
(Theorem~\ref{thm:gluedprod}) are $m$-goodness of the support and a condition on the labels, both of which are checked
pointwise. This is the route used for Theorem~\ref{thm:main2}.

In this section $n=2m$ with $m\ge3$ odd. We identify $C_{2m}^r=C_2^r\oplus C_m^r$ and write $\pi(g)\in\F_2^r$ (the
\emph{label}) and $\rho(g)\in(\Z/m\Z)^r$ for the components.

\begin{definition}\label{def:glued}
A sequence $S=\prod_{h\in H}h^{2m-1}\prod_{\ell\in L}\ell^{m-1}$ over $C_{2m}^r$ is \emph{glued} if the labels
$\pi(h)$, $h\in H$, are pairwise distinct, $\rho$ is injective on $H\cup L$, and $A=\rho(H\cup L)$ is $m$-good.
\end{definition}

\begin{lemma}[reduction]\label{lem:red}
Let $S$ be glued and let $W$ be a $2m$-term zero-sum with $\supp(W)\subseteq H\cup L$ and multiplicities $d(x)$. Then
either $d(x)\le m-1$ for all $x$, or $W=x^{2m}$, or $W=x^my^m$ with $x\ne y$ and $\pi(x)=\pi(y)$. Consequently $S$ has
no $2m$-term zero-sum subsequence if and only if there is no \emph{relation}: multiplicities $0\le d\le m-1$ on
$H\cup L$ with $\sum d=2m$, $\sum d(x)\rho(x)=0$ and $\sum_{d(x)\text{ odd}}\pi(x)=0$.
\end{lemma}

\begin{proof}
Let $e(x)\in[0,m-1]$ be $d(x)$ reduced modulo $m$. Then $\sum e(x)\rho(x)=0$ and $\sum e\in\{0,m,2m\}$. The value $m$
contradicts $m$-goodness; if $\sum e=2m$, then $e=d$; otherwise all $d(x)\in\{0,m,2m\}$. For $W=x^my^m$ the label
condition is $m(\pi(x)+\pi(y))=\pi(x)+\pi(y)=0$. Heavy terms have distinct labels and light terms have
multiplicity $m-1<m$, so the last two forms are not subsequences of $S$.
\end{proof}

\begin{lemma}\label{lem:prodgood}
If $A_1\subseteq C_m^{r_1}$ and $A_2\subseteq C_m^{r_2}$ are $m$-good, so is $A_1\times A_2$.
\end{lemma}

\begin{proof}
Project an $m$-term zero-sum of $(A_1\times A_2)^{m-1}$ to $A_1$. If all multiplicities of the projection are at most
$m-1$, this contradicts $A_1$ being $m$-good. Otherwise the projection is $x^m$, and the second coordinates form an
$m$-term zero-sum in $A_2$ with multiplicities at most $m-1$.
\end{proof}

\begin{theorem}[product of glued sequences]\label{thm:gluedprod}
Let $S_i=H_i^{2m-1}L_i^{m-1}$, $i=1,2$, be glued sequences without $2m$-term zero-sum. Then $S_1\otimes S_2$
(Definition~\ref{def:product}) is glued and has no $2m$-term zero-sum.
\end{theorem}

\begin{proof}
The heavy labels $(\pi(h_1),\pi(h_2))$ are distinct, and the support modulo $m$ is $m$-good by
Lemma~\ref{lem:prodgood}. Let $W$ be a $2m$-term zero-sum in the product, with projections $W_1,W_2$; each is a
$2m$-term zero-sum supported on $H_i\cup L_i$, so Lemma~\ref{lem:red} applies to it. The projections need not be
subsequences of $S_i$, since projection can increase multiplicities. If all multiplicities in $W_1$ are at most $m-1$,
then $W_1$ divides $S_1$, because every term of $S_1$ has multiplicity $2m-1$ or $m-1$; this is impossible. If
$W_1=x^{2m}$, then $W_2$ divides $S_2$. The same holds with the factors exchanged. It remains $W_1=x^mx'^m$ and
$W_2=y^my'^m$ with $\pi(x)=\pi(x')$ and $\pi(y)=\pi(y')$, so one of $x,x'$ and one of $y,y'$ is light. The
multiplicities of $(x,y),(x,y'),(x',y),(x',y')$ in $W$ are $a,m-a,m-a,a$. If $a\in\{0,m\}$, some term occurs $m$
times, so it is heavy$\times$heavy, i.e.\ $x$ and $x'$ are both heavy, which is impossible. If $0<a<m$, a
light$\times$light pair occurs, which is not a term of the product.
\end{proof}

\begin{lemma}\label{lem:WZglued}
For $n=2m$ the sequence $S_{WZ}(n)$ is glued.
\end{lemma}

\begin{proof}
If two heavy terms $x\ne y$ had the same residue modulo $2$, then $x^my^m$ would be a zero-sum subsequence of length
$2m$, since $m(x+y)\equiv0$ modulo $m$ and modulo $2$; so the heavy labels are distinct. Modulo $m$ the eighteen points
are $(y,x-z,x,t)$ with $t\in\{0,1\}$, except $t\in\{2,4\}$ in row B and $t\in\{1,3\}$ in row D. Points with different
labels differ in $(y,x-z,x)$, whose entries lie in $[-1,2]$, and within rows B and D the values of $t$ differ by $2$; so
$\rho$ is injective. For $m$-goodness, let an $m$-term zero-sum modulo $m$ with multiplicities at most $m-1$ be given.
Steps 1 and 2 of the proof of {\cite[Proposition~3.5]{WZ}} go through with $n$ replaced by $m$: they need $m\mid\sum y$,
$m\mid\sum(x-z)$, $m\mid\sum x$, and, on the third side, $m\mid\sum z$ together with $m$ odd. So all terms have the same
label. In that row the values of $t$ are $\{0,1\}$, $\{2,4\}$ or $\{1,3\}$, and $m\mid\sum t$ forces $m\mid j$ for the
number $j\le m-1$ of terms with the larger value; so one point is used $m$ times, a contradiction.
\end{proof}

\section{A rank-five block for \texorpdfstring{$n\equiv2\pmod 4$}{n = 2 mod 4}}\label{sec:rank5}

Let $V_4=\{z\in\{-1,0,1\}^4\setminus\{0\}:\ z_1^2+z_2^2+z_3^2\equiv z_4^2\pmod3\}$. It has $20$ points: the $12$
points $\pm e_i\pm e_4$ ($i\le3$) and the $8$ points $(\pm1,\pm1,\pm1,0)$. Reduced modulo $3$ it is a $20$-cap
\cite{EEGKR}.

\begin{lemma}\label{lem:V4}
$V_4$ is $m$-good for every odd $m\ge3$, and so is $A=\{0,1\}\times V_4\subseteq\Z^5$.
\end{lemma}

\begin{proof}
Let $\prod c_vv$, $0\le c_v\le m-1$, $\sum c_v=m$, be zero modulo $m$, so $\sum c_vv=my$ with $y\in\Z^4$ and $y\in\conv(V_4)$.
The form $z_1+z_2+z_3$ is odd on $V_4$, hence $m(y_1+y_2+y_3)\equiv\sum c_v\equiv m\pmod 2$, so $y_1+y_2+y_3$ is odd. The
lattice points of $\conv(V_4)$ with this property are $\pm e_i$ ($i\le3$) and the points of $V_4$ itself: indeed
$y\in\{-1,0,1\}^4$, since all points of $V_4$ are; and the inequality $|z_i|+|z_j|+|z_4|\le2$ ($i\ne j\le3$) holds on
$V_4$, hence on $\conv(V_4)$, so if $y_1,y_2,y_3$ are all nonzero then $y_4=0$ and $y\in V_4$, while if exactly one of
them is nonzero then $y=\pm e_i$ or $y=\pm e_i\pm e_4$. Every point of $V_4$ is the unique maximiser of a linear form
($\pm z_1\pm z_2\pm z_3$ or $\pm z_i\pm z_4$), hence a vertex; so $y\in V_4$ forces $c_y=m$. For $y=e_1$ (the other cases follow by symmetry) the support lies on the
facet $z_1=1$, whose points are $(1,0,0,\pm1)$ and $(1,\pm1,\pm1,0)$; the fourth coordinate forces equal multiplicities on
$(1,0,0,\pm1)$, and the second and third coordinates force $c_{(1,1,1,0)}=c_{(1,-1,-1,0)}$ and
$c_{(1,1,-1,0)}=c_{(1,-1,1,0)}$, so $m=\sum c_v$ is even, a contradiction. The second statement follows from
Lemma~\ref{lem:prodgood}, since $\{0,1\}$ is $m$-good.
\end{proof}

Table~\ref{tab:40} lists the $40$ points $a_i\in A$ (first coordinate $\varepsilon\in\{0,1\}$, then $b\in V_4$) with
labels $\pi_i\in\F_2^5$. The $32$ points marked H have pairwise distinct labels; the $8$ points marked L are
$(1,b)$ with $b=(\pm1,\pm1,\pm1,0)$, and each has the same label as $(0,b)$. For $n=2m$ let $g_i\in C_n^5$ be the
element with $g_i\equiv a_i\pmod m$ and $g_i\equiv\pi_i\pmod 2$, and let
\[
S_5(n)=\prod_{i\in\mathrm H}g_i^{\,n-1}\prod_{i\in\mathrm L}g_i^{\,m-1},\qquad |S_5(n)|=32(n-1)+8(m-1).
\]

\begin{table}[h]
\begin{center}\scriptsize\begin{tabular}{rccc|rccc}
\toprule
$i$ & $a_i$ & $\pi_i$ & & $i$ & $a_i$ & $\pi_i$ & \\
\midrule
1 & $(0,0,0,1,1)$ & $(0,0,0,0,0)$ & H & 21 & $(1,0,0,1,1)$ & $(0,1,1,1,0)$ & H\\
2 & $(0,0,0,1,-1)$ & $(0,0,0,1,0)$ & H & 22 & $(1,0,0,1,-1)$ & $(0,1,1,0,0)$ & H\\
3 & $(0,0,0,-1,1)$ & $(1,0,1,1,1)$ & H & 23 & $(1,0,0,-1,1)$ & $(1,1,0,0,1)$ & H\\
4 & $(0,0,0,-1,-1)$ & $(1,1,0,1,1)$ & H & 24 & $(1,0,0,-1,-1)$ & $(1,0,1,0,1)$ & H\\
5 & $(0,0,1,0,1)$ & $(0,1,0,1,1)$ & H & 25 & $(1,0,1,0,1)$ & $(0,0,1,0,1)$ & H\\
6 & $(0,0,1,0,-1)$ & $(0,1,0,0,1)$ & H & 26 & $(1,0,1,0,-1)$ & $(0,0,1,1,1)$ & H\\
7 & $(0,0,-1,0,1)$ & $(1,1,1,0,0)$ & H & 27 & $(1,0,-1,0,1)$ & $(1,1,1,1,0)$ & H\\
8 & $(0,0,-1,0,-1)$ & $(1,0,0,0,0)$ & H & 28 & $(1,0,-1,0,-1)$ & $(1,0,0,1,0)$ & H\\
9 & $(0,1,0,0,1)$ & $(0,1,0,1,0)$ & H & 29 & $(1,1,0,0,1)$ & $(0,0,1,1,0)$ & H\\
10 & $(0,1,0,0,-1)$ & $(0,0,1,0,0)$ & H & 30 & $(1,1,0,0,-1)$ & $(0,1,0,0,0)$ & H\\
11 & $(0,1,1,1,0)$ & $(1,0,1,0,0)$ & H & 31 & $(1,1,1,1,0)$ & $(1,0,1,0,0)$ & L\\
12 & $(0,1,1,-1,0)$ & $(0,1,1,0,1)$ & H & 32 & $(1,1,1,-1,0)$ & $(0,1,1,0,1)$ & L\\
13 & $(0,1,-1,1,0)$ & $(0,0,0,1,1)$ & H & 33 & $(1,1,-1,1,0)$ & $(0,0,0,1,1)$ & L\\
14 & $(0,1,-1,-1,0)$ & $(1,1,0,1,0)$ & H & 34 & $(1,1,-1,-1,0)$ & $(1,1,0,1,0)$ & L\\
15 & $(0,-1,0,0,1)$ & $(1,0,0,0,1)$ & H & 35 & $(1,-1,0,0,1)$ & $(1,0,0,1,1)$ & H\\
16 & $(0,-1,0,0,-1)$ & $(1,1,1,1,1)$ & H & 36 & $(1,-1,0,0,-1)$ & $(1,1,1,0,1)$ & H\\
17 & $(0,-1,1,1,0)$ & $(0,0,0,0,1)$ & H & 37 & $(1,-1,1,1,0)$ & $(0,0,0,0,1)$ & L\\
18 & $(0,-1,1,-1,0)$ & $(1,1,0,0,0)$ & H & 38 & $(1,-1,1,-1,0)$ & $(1,1,0,0,0)$ & L\\
19 & $(0,-1,-1,1,0)$ & $(1,0,1,1,0)$ & H & 39 & $(1,-1,-1,1,0)$ & $(1,0,1,1,0)$ & L\\
20 & $(0,-1,-1,-1,0)$ & $(0,1,1,1,1)$ & H & 40 & $(1,-1,-1,-1,0)$ & $(0,1,1,1,1)$ & L\\
\bottomrule
\end{tabular}\end{center}

\caption{The rank-five block: points $a_i=(\varepsilon,b)\in\{0,1\}\times V_4$, labels $\pi_i\in\F_2^5$ and roles.}
\label{tab:40}
\end{table}

\begin{theorem}\label{thm:rank5}
For every odd $m\ge3$ the sequence $S_5(2m)$ is glued and has no zero-sum subsequence of length $2m$.
\end{theorem}

The sequence is glued by Lemma~\ref{lem:V4}. By Lemma~\ref{lem:red} it suffices to exclude relations. Let $d$ be a
relation: $0\le d_i\le m-1$, $\sum d_i=2m$, $\sum d_ia_i=my$ with $y\in\Z^5$, and $\sum d_i\pi_i\equiv0\pmod 2$. Then
$y/2=\sum\lambda_ia_i$ with $\lambda_i=d_i/(2m)\in[0,\tfrac12)$ and $\sum\lambda_i=1$; in particular
$y\in\{0,1,2\}\times\{-2,\dots,2\}^4$. We use four elementary tools.

\begin{lemma}\label{lem:tools}
Let $d$ be a relation and $y$ as above.
\begin{enumerate}[label=\textup{(\roman*)}]
\item \emph{(Congruences.)} If $u\in\Z^5$, $w\in\Z$ and $\lambda\in\{0,1\}^5$ satisfy
$u\cdot a_i+w\equiv2\,\lambda\cdot\pi_i\pmod 4$ for all $i$ with $d_i>0$, then $u\cdot y+2w\equiv0\pmod4$.
\item \emph{(Faces.)} If $\varphi\in\Z^5$ and $\max_{i\in F}\varphi(a_i)=\varphi(y)/2$, where $F\supseteq\supp(d)$, then
$\supp(d)\subseteq\{i\in F:\varphi(a_i)=\varphi(y)/2\}$.
\item \emph{(Separation.)} If $\varphi(y)>\varphi_1+\varphi_2$, where $\varphi_1\ge\varphi_2$ are the two largest
values of $\varphi$ on $\{a_i: i\in F\}$ and $F\supseteq\supp(d)$, then there is no relation.
\item \emph{(Small support.)} $|\supp(d)|\ge3$.
\end{enumerate}
\end{lemma}

\begin{proof}
(i) Multiply by $d_i$ and sum: $m(u\cdot y)+2mw\equiv2\lambda\cdot\sum d_i\pi_i\equiv0\pmod 4$, and $m$ is odd. (ii)
$\sum\lambda_i(\varphi(y)/2-\varphi(a_i))=0$ with nonnegative summands. (iii) Since $0\le\lambda_i<\frac12$ and
$\sum\lambda_i=1$, $\varphi(y)/2=\sum\lambda_i\varphi(a_i)\le(\varphi_1+\varphi_2)/2$. (iv) $\sum d_i=2m$ and
$d_i\le m-1$.
\end{proof}

\begin{proof}[Proof of Theorem~\ref{thm:rank5}]
\emph{Global congruences.} All $40$ points satisfy
\[
b_3+1\equiv\pi_1+\pi_2+\pi_3,\qquad b_2\equiv\pi_1+\pi_5,\qquad b_1\equiv\pi_2+\pi_3+\pi_5\pmod 2,
\]
\[
b_1+b_2+b_3+3\equiv2\pi_1\pmod 4,
\]
where $a=(\varepsilon,b_1,b_2,b_3,b_4)$; these can be read off Table~\ref{tab:40}. By
Lemma~\ref{lem:tools}(i), $y=(y_0,y_1,\dots,y_4)$ has $y_1,y_2,y_3$ even and $y_1+y_2+y_3\equiv2\pmod 4$. This leaves the
set $\mathcal Y$ of $3\cdot14\cdot5=210$ vectors $y$ with $y_0\in\{0,1,2\}$, $y_4\in\{-2,\dots,2\}$, and
$(y_1,y_2,y_3)\in\{-2,0,2\}^3$ with an odd number of nonzero entries.

\emph{The $210$ vectors.} For each $y\in\mathcal Y$ we have a certificate of the following form. A list of linear
forms $\varphi$ is applied successively by Lemma~\ref{lem:tools}(ii), starting from $F$ = all $40$ points and replacing
$F$ by the face each time. Then one of the following applies: a separating form (Lemma~\ref{lem:tools}(iii),
$136$ vectors); a face with at most two points (Lemma~\ref{lem:tools}(iv), $20$ vectors); or a congruence
$u\cdot a_i+w\equiv2\lambda\cdot\pi_i\pmod 4$ valid on the final face $F$ and violated by $y$ in the sense of
Lemma~\ref{lem:tools}(i) ($54$ vectors). All forms have entries in $\{-2,\dots,2\}$, and at most three faces are needed.
None of these arguments depends on $m$. The $210$ certificates are listed in Table~\ref{tab:cert} in the
appendix; each one is a finite check from Table~\ref{tab:40}.
\end{proof}

\emph{A worked example.} Take $y=(0,-2,0,0,0)$. The form $\varphi=(-1,0,0,0,0)$ has maximum $0=\varphi(y)/2$ on
the $40$ points, so by Lemma~\ref{lem:tools}(ii) the support lies on the face $\varepsilon=0$. On it, $\varphi=(0,-1,0,0,0)$
has maximum $1=\varphi(y)/2$, so the support lies in $F=\{15,\dots,20\}$, the six points $(0,-1,b_2,b_3,b_4)$ of
Table~\ref{tab:40}. On $F$ the congruence $b_3+b_4+1\equiv2(\pi_4+\pi_5)\pmod4$ holds (both sides are $2,0,2,0,2,0$ for
$i=15,\dots,20$), so Lemma~\ref{lem:tools}(i) with $u=(0,0,0,1,1)$, $w=1$, $\lambda=(0,0,0,1,1)$ gives
$y_3+y_4+2\equiv0\pmod4$; but $y_3+y_4+2=2$. So there is no relation with this $y$. In Table~\ref{tab:cert} this is the
row ``$0\bar2000$; $\bar10000,0\bar1000$; $\mathrm C\ 00011|1|00011$''.

The certificates were found by computer and checked independently; a machine-readable version of Table~\ref{tab:cert}
is available from the author. For $n=6,10,14,18,22$ we also confirmed directly that $S_5(n)$ has no $n$-term zero-sum.

\begin{proof}[Proof of Theorem~\ref{thm:main2}]
(a) Theorem~\ref{thm:rank5} and Lemma~\ref{lem:eta}. (b) By Theorem~\ref{thm:rank5} and Theorem~\ref{thm:gluedprod}
(applied repeatedly; the product of glued sequences is glued), the product of $\lfloor r/5\rfloor$ copies of $S_5(n)$
has no $n$-term zero-sum; its light density is $\lfloor r/5\rfloor/4$. By \eqref{eq:prodlength} and doubling its length
over $C_n^r$ is $2^r(n-1)+(m-1)2^{r-2}\lfloor r/5\rfloor$. (c) In the same way, with $a$ copies of $S_5(n)$ and $b$
copies of $S_{WZ}(n)$, which is glued by Lemma~\ref{lem:WZglued}; the light density is $a/4+b/8$. In all cases
Lemma~\ref{lem:eta} gives the bound for $\eta$.
\end{proof}

\begin{remark}\label{rem:n6}
For $n=6$ the first version of this paper gave explicit sequences of lengths $84$, $176$ and $352$ in ranks $4$, $5$
and $6$, found by a computer search on complete caps of sizes $18$, $40$ and $80$; Theorem~\ref{thm:main2} now gives the same
lengths for all $n\equiv2\pmod4$ by a uniform proof. A maximum cap cannot be used as a support: for $m=3$, in a glued sequence whose $2^r$ heavy labels exhaust $\F_2^r$, each
light term must form a pair with a heavy term whose sum modulo $3$ is not the sum of any other pair, and the
$20$-cap in $AG(4,3)$ and Edel's $42$-point and $96$-point sets \cite{Edel04b} contain no such pair.
\end{remark}

\section{Zero is the last \texorpdfstring{$n$}{n}-term sum}\label{sec:last}

For a sequence $S$ over $G$ with $\exp(G)=n$ let $\Sigma_n(S)$ be the set of sums of subsequences of length $n$. If $S$
has no zero-sum subsequence of length $n$, then $0\notin\Sigma_n(S)$, so $|\Sigma_n(S)|\le|G|-1$.

\begin{conjecture}\label{conj:last}
Let $G=C_n^r$, and let $S$ be a sequence over $G$ of length $\s(G)-1$ without zero-sum subsequence of length $n$. Then
$\Sigma_n(S)=G\setminus\{0\}$; that is, the $n$-term subsequence sums of $S$ cover all $|G|-1$ nonzero elements.
\end{conjecture}

This concerns concrete extremal sequences and is much weaker than a full inverse result. The set $\Sigma_n(S)$ is
invariant under translations of $S$ (which change all $n$-term sums by $ng=0$) and is mapped bijectively by
automorphisms of $G$. Call a sequence $S$ without $n$-term zero-sum \emph{maximal} (that is, inclusion-maximal, not
necessarily of maximum length) if $Sx$ has an $n$-term zero-sum for every $x\in G$; equivalently, $\Sigma_{n-1}(S)=G$ (the zero-sum of $Sx$ must contain $x$). Every sequence of length
$\s(G)-1$ without $n$-term zero-sum is maximal. The conjecture holds as soon as one term has the largest possible
multiplicity.

\begin{proposition}\label{prop:lastmax}
Let $\exp(G)=n$ and let $S$ be a maximal sequence over $G$ without $n$-term zero-sum. If some term $t$ of $S$ has
multiplicity $n-1$, then $\Sigma_n(S)=G\setminus\{0\}$.
\end{proposition}

\begin{proof}
Let $g\ne0$. Since $\Sigma_{n-1}(S)=G$ there is a subsequence $U$ of $S$ of length $n-1$ with sum $g-t$. If
$U=t^{n-1}$, then $g-t=(n-1)t=-t$, so $g=0$, which is excluded. Hence $U$ contains at most $n-2$ copies of $t$, and
$Ut$ is a subsequence of $S$ of length $n$ with sum $g$.
\end{proof}

\begin{corollary}\label{cor:last}
Conjecture~\ref{conj:last} holds for every extremal sequence that contains a term of multiplicity $n-1$. In
particular it holds for all extremal sequences of the form $T^{n-1}$, hence for all groups $C_n^r$ with Property D,
for $r=1$ and for $n\in\{2,3\}$, and for all the sequences of this paper and of \cite{WZ} whenever they are extremal.
\end{corollary}

Thus a sufficient condition for Conjecture~\ref{conj:last} is the following weak form of Property D, which is open
for all groups $C_n^r$ with $r\ge2$ where Property D is not known: \emph{every sequence of length $\s(C_n^r)-1$ without $n$-term
zero-sum contains a term of multiplicity $n-1$}. For even $n$ this is exactly the hypothesis of Lemma~\ref{lem:eta},
and all sequences of Theorems~\ref{thm:main} and \ref{thm:main2} satisfy it. We add one observation on the
remaining case.

\begin{lemma}\label{lem:lastrepl}
Let $S$ be a sequence of length $\s(G)-1$ without $n$-term zero-sum, and let $g\ne0$ with $g\notin\Sigma_n(S)$. Then
for every term $t$ of $S$ the sequence $St^{-1}(t-g)$, obtained by replacing one copy of $t$ by $t-g$, is again a
sequence of length $\s(G)-1$ without $n$-term zero-sum.
\end{lemma}

\begin{proof}
An $n$-term zero-sum of $St^{-1}(t-g)$ must contain $t-g$, so it is $U(t-g)$ with $U$ a subsequence of $St^{-1}$ of
length $n-1$ and sum $g-t$. Then $Ut$ is a subsequence of $S$ (as $U$ uses at most $v_t(S)-1$ copies of $t$) of length
$n$ and sum $g$, contradicting $g\notin\Sigma_n(S)$.
\end{proof}

So a counterexample to the conjecture would come with a whole family of extremal sequences; a single extremal
sequence without a term of multiplicity $n-1$ would be a counterexample to the weak form of Property D above.

\emph{Numerical checks.} For the following sequences $S$ without $n$-term zero-sum we verified by computer that
$\Sigma_{n-1}(S)=G$, i.e.\ that $S$ is maximal (so that Proposition~\ref{prop:lastmax} applies), and also directly
that $\Sigma_n(S)=G\setminus\{0\}$.
\begin{itemize}
\item $T_0^{n-1}$ over $C_n^3$, with the $9$-point set $T_0$ of \cite{Els04}: $n=3,5,7,9,11,13,15,17,19,21,25$.
\item $V_4^{n-1}$ over $C_n^4$, with the $20$-point set $V_4$ of Section~\ref{sec:rank5}: $n=3,5,7,9,11,13,15$; the
$20$-point set of \cite{Els16} for $n=5,7,9,11,13$; and Edel's $20$-point set \cite{Edel04b} for $n=15$.
\item $E_5^{n-1}$ over $C_n^5$, with Edel's $42$-point set \cite{Edel04b}: $n=3,5,7,9,11,13,15$; and Edel's $96$-point set
over $C_n^6$: $n=3,5,7$.
\item $S_{WZ}(n)$ over $C_n^4$: $n=6,10,12,14,18,20$; and $S_5(n)$ over $C_n^5$: $n=6,10,14$.
\end{itemize}
For the glued sequences $S_{WZ}(n)$ and $S_5(n)$ the heavy part alone already covers all nonzero elements. The
sequences $T_0^{n-1}$ and $V_4^{n-1}$ are extremal whenever $\s(C_n^r)=c_r(n-1)+1$ is known, for instance for $n=3,9$,
and in rank $3$ also for $n=5,15$ (Section~\ref{sec:div});
in general extremality is not known, and maximality is what the computation establishes. The property fails for
short sequences: for five points of $V_4$ and $n=5$, only $78$ of the $625$ elements are covered.

\section{The role of computation}\label{sec:deps}

All theorems are lower bounds proved by explicit constructions. Computers were used to find the constructions and
the certificates. The proofs are of the following kinds.
\begin{itemize}
\item \emph{Analytic:} Lemmas~\ref{lem:basic}--\ref{lem:double}, \ref{lem:div}, \ref{lem:red}, \ref{lem:prodgood},
\ref{lem:WZglued}, \ref{lem:V4}, \ref{lem:tools}; Proposition~\ref{prop:equiv}, Proposition~\ref{prop:WZrigid};
Theorems~\ref{thm:rigidprod} and \ref{thm:gluedprod}; hence Theorem~\ref{thm:main} and Corollary~\ref{cor:growth};
Proposition~\ref{prop:lastmax}, Corollary~\ref{cor:last} and Lemma~\ref{lem:lastrepl}.
\item \emph{By an explicit case study:} Theorem~\ref{thm:rank5}, via the four congruences and the $210$ certificates of Table~\ref{tab:cert},
each of which is a finite check independent of $n$. Hence Theorem~\ref{thm:main2}.
\item \emph{Checks not used in proofs:} the direct verifications mentioned after
Proposition~\ref{prop:WZrigid} and Theorem~\ref{thm:rank5}.
\end{itemize}

\section{Concluding remarks}

\begin{enumerate}[label=(\arabic*),leftmargin=*]
\item \emph{Powers of two versus other even exponents.} For $n=2^a$, Harborth's bound is exact. For every other even
$n$ the excess $\s(C_n^r)-1-2^r(n-1)$ is at least of order $(m-1)\,r\,2^r$. The subgroup bound
$\s(C_{2m'}^r)\le m'2^r+\s(C_{m'}^r)$ (see \cite{GG} and {\cite[Section~4]{EEGKR}}) shows that, for $n=2m$, the excess is
at most $\s(C_m^r)-1-(m-1)2^r$. This upper bound is exponentially larger than $(m-1)2^r$: Edel's $96$-point set in
rank $6$, which is $m$-good for every odd $m$ \cite{Edel04b}, and Lemma~\ref{lem:prodgood} give
$\s(C_m^r)-1\ge(m-1)\,96^{\lfloor r/6\rfloor}2^{r-6\lfloor r/6\rfloor}\ge(m-1)\,c\cdot2.14^r$.
\item \emph{Small ranks.} Let $n=2q$ with $q$ an odd prime. If $\s(C_q^4)=20q-19$ and $\s(C_q^5)=42q-41$, as
conjectured, the subgroup bound gives $\s(C_{2q}^4)-1\le36q-20$ and $\s(C_{2q}^5)-1\le74q-42$, while Theorems~\ref{thm:WZ}
and \ref{thm:main2}(a) give $34q-18$ and $72q-40$. In both ranks the gap is $2(q-1)$, the length of two light terms. Since at most $16$ terms of a sequence
over $C_{2q}^4$ without $2q$-term zero-sum can have multiplicity at least $q$ (their residues modulo $2$ are distinct),
and all other terms have multiplicity at most $q-1$, closing the gap in rank $4$ requires a support of at least $20$
points; a $19$-point support could improve the bound by one light term only. The sequences of Wang and Zeng cannot be extended, and the $20$-point sets of \cite{EEGKR,Edel04b,Els16} admit no light
term at all for $m=3$ (Remark~\ref{rem:n6}) and, by a computation of all relations, also for $m=5$ and $m=7$. In rank $3$ the subgroup bound $17q-9$ exceeds Harborth's $16q-8$ by a single light term;
a search for $n=6$ and $n=10$ found no such sequence. In rank $5$ a $42$-point support would be needed; we were not able
to find one for $q=5$ or $q=7$ by modifying, doubling and twisting the $20$-point sets.
\item \emph{A rank-five block for $n\equiv0\pmod4$.} For $n\equiv0\pmod 4$ we only have the rank-four block of Wang and
Zeng. A rank-five block with light density $1/4$ would improve Theorem~\ref{thm:main} also for these $n$. The direct
analogue of $S_5(n)$, with labels $0/1$ in $C_4^5$, has zero-sums for $n=12$ and $n=20$.
\item \emph{Larger blocks.} In rank six, light density $1/4$ (as for $n=6$ with $352$ terms) is no better than products
of rank-five blocks. Better rates would need blocks of light density larger than $r/20$ in rank $r$.
\item \emph{Odd exponents.} For odd $n$ the parity obstruction is absent, and no counterexample to Properties C or D
is known. Proposition~\ref{prop:equiv} locates where one would have to look: if $M_n(r)$ is known, then Property~D for
$C_n^r$ forces $\s(C_n^r)-1=(n-1)M_n(r)$, so \emph{any} sequence of length $(n-1)M_n(r)$ without $n$-term zero-sum that
is not of the form $T^{n-1}$ disproves Property D, whatever the value of $\s(C_n^r)$. By Corollary~\ref{cor:values}(3),
$M_n(3)=9$ and $M_n(4)=20$ for all odd $n$ divisible by $3$. The natural candidates are sequences in which one heavy term
$h^{n-1}$ of $T^{n-1}$ is replaced by $h^{a}\ell^{\,n-1-a}$ with a new element $\ell$; for $T=T_0$ in rank $3$ and
$n=9,15,21,25$, and for the $20$-point sets $V_4$ and that of \cite{Els16} in rank $4$ and $n=9,15$, an exhaustive
search shows that no such replacement is possible. Simulated annealing over sequences of length $126$ over $C_{15}^3$ with other multiplicity
profiles (eight heavy terms and lights of multiplicities $4,4,4,2$; seven heavy terms and seven of multiplicity $4$) did not
find a sequence without $15$-term zero-sum either; this is evidence about these searches, not about Property D. In
rank $3$ the remaining question is therefore whether Property D holds for $C_{15}^3$. Since $\s(C_{15}^3)=127$ is known
(Section~\ref{sec:div}), this concerns the structure of the sequences of length $126$ rather than the value of the
constant.
\end{enumerate}

\appendix
\section{The certificates for Theorem~\ref{thm:rank5}}\label{app:cert}

Table~\ref{tab:cert} lists, for each of the $210$ vectors $y\in\mathcal Y$, the certificate used in the proof of
Theorem~\ref{thm:rank5}. Vectors are written as strings of their five coordinates, with $\bar k=-k$. The column
``faces'' gives the linear forms $\varphi$ applied successively by Lemma~\ref{lem:tools}(ii), starting from all $40$
points of Table~\ref{tab:40}. In the last column, $\mathrm S\ \varphi$ is a separating form
(Lemma~\ref{lem:tools}(iii)); $\mathrm P$ means that the final face has at most two points (Lemma~\ref{lem:tools}(iv));
and $\mathrm C\ u|w|\lambda$ is a congruence $u\cdot a_i+w\equiv2\lambda\cdot\pi_i\pmod4$ valid on the final face and
violated by $y$ (Lemma~\ref{lem:tools}(i)).

\begin{table}[p]
{\fontsize{5.4}{5.95}\selectfont\setlength{\tabcolsep}{0pt}\renewcommand{\arraystretch}{1}
\centering\begin{tabular}[t]{@{}l@{\hspace{5pt}}l@{\hspace{5pt}}l@{}}
$y$ & faces $\varphi$ & final\\ \hline
$0\bar{2}\bar{2}\bar{2}\bar{2}$ & $\bar{1}0000,\,0\bar{1}000,\,00\bar{1}00$ & $\mathrm S\ 000\bar{1}0$\\
$0\bar{2}\bar{2}\bar{2}\bar{1}$ & $\bar{1}0000,\,0\bar{1}000,\,00\bar{1}00$ & $\mathrm S\ 000\bar{1}0$\\
$0\bar{2}\bar{2}\bar{2}0$ & $\bar{1}0000,\,0\bar{1}000,\,00\bar{1}00$ & $\mathrm S\ 000\bar{1}0$\\
$0\bar{2}\bar{2}\bar{2}1$ & $\bar{1}0000,\,0\bar{1}000,\,00\bar{1}00$ & $\mathrm S\ 000\bar{1}0$\\
$0\bar{2}\bar{2}\bar{2}2$ & $\bar{1}0000,\,0\bar{1}000,\,00\bar{1}00$ & $\mathrm S\ 000\bar{1}0$\\
$0\bar{2}\bar{2}2\bar{2}$ & $\bar{1}0000,\,0\bar{1}000,\,00\bar{1}00$ & $\mathrm S\ 0000\bar{1}$\\
$0\bar{2}\bar{2}2\bar{1}$ & $\bar{1}0000,\,0\bar{1}000,\,00\bar{1}00$ & $\mathrm S\ 0000\bar{1}$\\
$0\bar{2}\bar{2}20$ & $\bar{1}0000,\,0\bar{1}000,\,00\bar{1}00$ & $\mathrm S\ 00010$\\
$0\bar{2}\bar{2}21$ & $\bar{1}0000,\,0\bar{1}000,\,00\bar{1}00$ & $\mathrm S\ 00001$\\
$0\bar{2}\bar{2}22$ & $\bar{1}0000,\,0\bar{1}000,\,00\bar{1}00$ & $\mathrm S\ 00001$\\
$0\bar{2}00\bar{2}$ & $\bar{1}0000,\,0\bar{1}000$ & $\mathrm S\ 0000\bar{1}$\\
$0\bar{2}00\bar{1}$ & $\bar{1}0000,\,0\bar{1}000$ & $\mathrm C\ 00002|0|11010$\\
$0\bar{2}000$ & $\bar{1}0000,\,0\bar{1}000$ & $\mathrm C\ 00011|1|00011$\\
$0\bar{2}001$ & $\bar{1}0000,\,0\bar{1}000$ & $\mathrm C\ 00002|0|11010$\\
$0\bar{2}002$ & $\bar{1}0000,\,0\bar{1}000$ & $\mathrm S\ 00001$\\
$0\bar{2}2\bar{2}\bar{2}$ & $\bar{1}0000,\,0\bar{1}000,\,000\bar{1}0$ & $\mathrm S\ 0000\bar{1}$\\
$0\bar{2}2\bar{2}\bar{1}$ & $\bar{1}0000,\,0\bar{1}000,\,000\bar{1}0$ & $\mathrm S\ 0000\bar{1}$\\
$0\bar{2}2\bar{2}0$ & $\bar{1}0000,\,0\bar{1}000,\,000\bar{1}0$ & $\mathrm S\ 00100$\\
$0\bar{2}2\bar{2}1$ & $\bar{1}0000,\,0\bar{1}000,\,000\bar{1}0$ & $\mathrm S\ 00001$\\
$0\bar{2}2\bar{2}2$ & $\bar{1}0000,\,0\bar{1}000,\,000\bar{1}0$ & $\mathrm S\ 00001$\\
$0\bar{2}22\bar{2}$ & $\bar{1}0000,\,0\bar{1}000$ & $\mathrm S\ 0000\bar{1}$\\
$0\bar{2}22\bar{1}$ & $\bar{1}0000,\,0\bar{1}000,\,00010$ & $\mathrm S\ 0000\bar{1}$\\
$0\bar{2}220$ & $\bar{1}0000,\,0\bar{1}000,\,00010$ & $\mathrm S\ 00100$\\
$0\bar{2}221$ & $\bar{1}0000,\,0\bar{1}000,\,00010$ & $\mathrm S\ 00001$\\
$0\bar{2}222$ & $\bar{1}0000,\,0\bar{1}000$ & $\mathrm S\ 00001$\\
$00\bar{2}0\bar{2}$ & $\bar{1}0000,\,00\bar{1}00$ & $\mathrm S\ 0000\bar{1}$\\
$00\bar{2}0\bar{1}$ & $\bar{1}0000,\,00\bar{1}00$ & $\mathrm C\ 00002|0|10011$\\
$00\bar{2}00$ & $\bar{1}0000,\,00\bar{1}00$ & $\mathrm C\ 00011|1|00101$\\
$00\bar{2}01$ & $\bar{1}0000,\,00\bar{1}00$ & $\mathrm C\ 00002|0|10011$\\
$00\bar{2}02$ & $\bar{1}0000,\,00\bar{1}00$ & $\mathrm S\ 00001$\\
$000\bar{2}\bar{2}$ & $\bar{1}0000,\,000\bar{1}0$ & $\mathrm S\ 0000\bar{1}$\\
$000\bar{2}\bar{1}$ & $\bar{1}0000,\,000\bar{1}0$ & $\mathrm C\ 00002|2|01101$\\
$000\bar{2}0$ & $\bar{1}0000,\,000\bar{1}0$ & $\mathrm C\ 00101|1|01010$\\
$000\bar{2}1$ & $\bar{1}0000,\,000\bar{1}0$ & $\mathrm C\ 00002|2|01101$\\
$000\bar{2}2$ & $\bar{1}0000,\,000\bar{1}0$ & $\mathrm S\ 00001$\\
$0002\bar{2}$ & $\bar{1}0000,\,0000\bar{1}$ & $\mathrm S\ 00010$\\
$0002\bar{1}$ & $\bar{1}0000,\,00010$ & $\mathrm C\ 00002|2|00101$\\
$00020$ & $\bar{1}0000,\,00010$ & $\mathrm C\ 00101|3|00010$\\
$00021$ & $\bar{1}0000,\,00010$ & $\mathrm C\ 00002|2|00101$\\
$00022$ & $\bar{1}0000,\,00001$ & $\mathrm S\ 00010$\\
$0020\bar{2}$ & $\bar{1}0000,\,0000\bar{1}$ & $\mathrm S\ 00100$\\
$0020\bar{1}$ & $\bar{1}0000,\,00100$ & $\mathrm C\ 00002|0|11100$\\
$00200$ & $\bar{1}0000,\,00100$ & $\mathrm C\ 00011|1|11011$\\
$00201$ & $\bar{1}0000,\,00100$ & $\mathrm C\ 00002|0|11100$\\
$00202$ & $\bar{1}0000,\,00001$ & $\mathrm S\ 00100$\\
$02\bar{2}\bar{2}\bar{2}$ & $\bar{1}0000,\,00\bar{1}00,\,000\bar{1}0$ & $\mathrm S\ 0000\bar{1}$\\
$02\bar{2}\bar{2}\bar{1}$ & $\bar{1}0000,\,00\bar{1}00,\,000\bar{1}0$ & $\mathrm S\ 0000\bar{1}$\\
$02\bar{2}\bar{2}0$ & $\bar{1}0000,\,00\bar{1}00,\,000\bar{1}0$ & $\mathrm S\ 01000$\\
$02\bar{2}\bar{2}1$ & $\bar{1}0000,\,00\bar{1}00,\,000\bar{1}0$ & $\mathrm S\ 00001$\\
$02\bar{2}\bar{2}2$ & $\bar{1}0000,\,00\bar{1}00,\,000\bar{1}0$ & $\mathrm S\ 00001$\\
$02\bar{2}2\bar{2}$ & $\bar{1}0000,\,00\bar{1}00$ & $\mathrm S\ 0000\bar{1}$\\
$02\bar{2}2\bar{1}$ & $\bar{1}0000,\,00\bar{1}00,\,00010$ & $\mathrm S\ 0000\bar{1}$\\
$02\bar{2}20$ & $\bar{1}0000,\,00\bar{1}00,\,00010$ & $\mathrm S\ 01000$\\
$02\bar{2}21$ & $\bar{1}0000,\,00\bar{1}00,\,00010$ & $\mathrm S\ 00001$\\
$02\bar{2}22$ & $\bar{1}0000,\,00\bar{1}00$ & $\mathrm S\ 00001$\\
$0200\bar{2}$ & $\bar{1}0000,\,0000\bar{1}$ & $\mathrm S\ 01000$\\
$0200\bar{1}$ & $\bar{1}0000,\,01000$ & $\mathrm C\ 00002|0|10111$\\
$02000$ & $\bar{1}0000,\,01000$ & $\mathrm C\ 00011|1|10010$\\
$02001$ & $\bar{1}0000,\,01000$ & $\mathrm C\ 00002|0|10111$\\
$02002$ & $\bar{1}0000,\,00001$ & $\mathrm S\ 01000$\\
$022\bar{2}\bar{2}$ & $\bar{1}0000,\,000\bar{1}0$ & $\mathrm S\ 0000\bar{1}$\\
$022\bar{2}\bar{1}$ & $\bar{1}0000,\,000\bar{1}0,\,00100$ & $\mathrm S\ 0000\bar{1}$\\
$022\bar{2}0$ & $\bar{1}0000,\,000\bar{1}0,\,00100$ & $\mathrm S\ 01000$\\
$022\bar{2}1$ & $\bar{1}0000,\,000\bar{1}0,\,00100$ & $\mathrm S\ 00001$\\
$022\bar{2}2$ & $\bar{1}0000,\,000\bar{1}0$ & $\mathrm S\ 00001$\\
$0222\bar{2}$ & $\bar{1}0000,\,0000\bar{1}$ & $\mathrm S\ 00010$\\
$0222\bar{1}$ & $\bar{1}0000,\,00010,\,00100$ & $\mathrm S\ 0000\bar{1}$\\
$02220$ & $\bar{1}0000,\,00010,\,00100$ & $\mathrm S\ 01000$\\
$02221$ & $\bar{1}0000,\,00010,\,00100$ & $\mathrm S\ 00001$\\
$02222$ & $\bar{1}0000,\,00001$ & $\mathrm S\ 00010$\\
$1\bar{2}\bar{2}\bar{2}\bar{2}$ & $0\bar{1}000,\,00\bar{1}00,\,000\bar{1}0$ & $\mathrm S\ 0000\bar{1}$\\
$1\bar{2}\bar{2}\bar{2}\bar{1}$ & $0\bar{1}000,\,00\bar{1}00,\,000\bar{1}0$ & $\mathrm S\ 0000\bar{1}$\\
$1\bar{2}\bar{2}\bar{2}0$ & $0\bar{1}000,\,00\bar{1}00,\,000\bar{1}0$ & $\mathrm P$\\
$1\bar{2}\bar{2}\bar{2}1$ & $0\bar{1}000,\,00\bar{1}00,\,000\bar{1}0$ & $\mathrm S\ 00001$\\
$1\bar{2}\bar{2}\bar{2}2$ & $0\bar{1}000,\,00\bar{1}00,\,000\bar{1}0$ & $\mathrm S\ 00001$\\
$1\bar{2}\bar{2}2\bar{2}$ & $0\bar{1}000,\,00\bar{1}00$ & $\mathrm S\ 0000\bar{1}$\\
$1\bar{2}\bar{2}2\bar{1}$ & $0\bar{1}000,\,00\bar{1}00$ & $\mathrm S\ 0000\bar{1}$\\
$1\bar{2}\bar{2}20$ & $0\bar{1}000,\,00\bar{1}00,\,00010$ & $\mathrm P$\\
$1\bar{2}\bar{2}21$ & $0\bar{1}000,\,00\bar{1}00$ & $\mathrm S\ 00001$\\
$1\bar{2}\bar{2}22$ & $0\bar{1}000,\,00\bar{1}00$ & $\mathrm S\ 00001$\\
$1\bar{2}00\bar{2}$ & $0\bar{1}000,\,0000\bar{1}$ & $\mathrm P$\\
$1\bar{2}00\bar{1}$ & $0\bar{1}000$ & $\mathrm C\ 00002|0|11100$\\
$1\bar{2}000$ & $0\bar{1}000$ & $\mathrm C\ 00011|1|00101$\\
$1\bar{2}001$ & $0\bar{1}000$ & $\mathrm C\ 00002|0|11100$\\
$1\bar{2}002$ & $0\bar{1}000,\,00001$ & $\mathrm P$\\
$1\bar{2}2\bar{2}\bar{2}$ & $0\bar{1}000,\,000\bar{1}0$ & $\mathrm S\ 0000\bar{1}$\\
$1\bar{2}2\bar{2}\bar{1}$ & $0\bar{1}000,\,000\bar{1}0$ & $\mathrm S\ 0000\bar{1}$\\
$1\bar{2}2\bar{2}0$ & $0\bar{1}000,\,000\bar{1}0,\,00100$ & $\mathrm P$\\
$1\bar{2}2\bar{2}1$ & $0\bar{1}000,\,000\bar{1}0$ & $\mathrm S\ 00001$\\
$1\bar{2}2\bar{2}2$ & $0\bar{1}000,\,000\bar{1}0$ & $\mathrm S\ 00001$\\
$1\bar{2}22\bar{2}$ & $0\bar{1}000,\,0000\bar{1}$ & $\mathrm S\ 00010$\\
$1\bar{2}22\bar{1}$ & $0\bar{1}000,\,00010$ & $\mathrm S\ 0000\bar{1}$\\
$1\bar{2}220$ & $0\bar{1}000,\,00010,\,00100$ & $\mathrm P$\\
$1\bar{2}221$ & $0\bar{1}000,\,00010$ & $\mathrm S\ 00001$\\
$1\bar{2}222$ & $0\bar{1}000,\,00001$ & $\mathrm S\ 00010$\\
$10\bar{2}0\bar{2}$ & $00\bar{1}00,\,0000\bar{1}$ & $\mathrm P$\\
$10\bar{2}0\bar{1}$ & $00\bar{1}00$ & $\mathrm C\ 00002|0|11100$\\
$10\bar{2}00$ & $00\bar{1}00$ & $\mathrm C\ 00011|1|00101$\\
$10\bar{2}01$ & $00\bar{1}00$ & $\mathrm C\ 00002|0|11100$\\
$10\bar{2}02$ & $00\bar{1}00,\,00001$ & $\mathrm P$\\
$100\bar{2}\bar{2}$ & $000\bar{1}0,\,0000\bar{1}$ & $\mathrm P$\\
$100\bar{2}\bar{1}$ & $000\bar{1}0$ & $\mathrm C\ 00002|2|01101$\\
$100\bar{2}0$ & $000\bar{1}0$ & $\mathrm C\ 00101|1|01010$\\
$100\bar{2}1$ & $000\bar{1}0$ & $\mathrm C\ 00002|2|01101$\\
$100\bar{2}2$ & $000\bar{1}0,\,00001$ & $\mathrm P$\\
\end{tabular}\hspace{2.5em}\begin{tabular}[t]{@{}l@{\hspace{5pt}}l@{\hspace{5pt}}l@{}}
$y$ & faces $\varphi$ & final\\ \hline
$1002\bar{2}$ & $0000\bar{1},\,00010$ & $\mathrm P$\\
$1002\bar{1}$ & $00010$ & $\mathrm C\ 00002|2|01101$\\
$10020$ & $00010$ & $\mathrm C\ 00101|3|01010$\\
$10021$ & $00010$ & $\mathrm C\ 00002|2|01101$\\
$10022$ & $00001,\,00010$ & $\mathrm P$\\
$1020\bar{2}$ & $0000\bar{1},\,00100$ & $\mathrm P$\\
$1020\bar{1}$ & $00100$ & $\mathrm C\ 00002|0|11100$\\
$10200$ & $00100$ & $\mathrm C\ 00011|1|11011$\\
$10201$ & $00100$ & $\mathrm C\ 00002|0|11100$\\
$10202$ & $00001,\,00100$ & $\mathrm P$\\
$12\bar{2}\bar{2}\bar{2}$ & $00\bar{1}00,\,000\bar{1}0$ & $\mathrm S\ 0000\bar{1}$\\
$12\bar{2}\bar{2}\bar{1}$ & $00\bar{1}00,\,000\bar{1}0$ & $\mathrm S\ 0000\bar{1}$\\
$12\bar{2}\bar{2}0$ & $00\bar{1}00,\,000\bar{1}0,\,01000$ & $\mathrm P$\\
$12\bar{2}\bar{2}1$ & $00\bar{1}00,\,000\bar{1}0$ & $\mathrm S\ 00001$\\
$12\bar{2}\bar{2}2$ & $00\bar{1}00,\,000\bar{1}0$ & $\mathrm S\ 00001$\\
$12\bar{2}2\bar{2}$ & $00\bar{1}00,\,0000\bar{1}$ & $\mathrm S\ 00010$\\
$12\bar{2}2\bar{1}$ & $00\bar{1}00,\,00010$ & $\mathrm S\ 0000\bar{1}$\\
$12\bar{2}20$ & $00\bar{1}00,\,00010,\,01000$ & $\mathrm P$\\
$12\bar{2}21$ & $00\bar{1}00,\,00010$ & $\mathrm S\ 00001$\\
$12\bar{2}22$ & $00\bar{1}00,\,00001$ & $\mathrm S\ 00010$\\
$1200\bar{2}$ & $0000\bar{1},\,01000$ & $\mathrm P$\\
$1200\bar{1}$ & $01000$ & $\mathrm C\ 00002|0|11100$\\
$12000$ & $01000$ & $\mathrm C\ 00011|1|10010$\\
$12001$ & $01000$ & $\mathrm C\ 00002|0|11100$\\
$12002$ & $00001,\,01000$ & $\mathrm P$\\
$122\bar{2}\bar{2}$ & $000\bar{1}0,\,0000\bar{1}$ & $\mathrm S\ 00100$\\
$122\bar{2}\bar{1}$ & $000\bar{1}0,\,00100$ & $\mathrm S\ 0000\bar{1}$\\
$122\bar{2}0$ & $000\bar{1}0,\,00100,\,01000$ & $\mathrm P$\\
$122\bar{2}1$ & $000\bar{1}0,\,00100$ & $\mathrm S\ 00001$\\
$122\bar{2}2$ & $000\bar{1}0,\,00001$ & $\mathrm S\ 00100$\\
$1222\bar{2}$ & $0000\bar{1},\,00010$ & $\mathrm S\ 00100$\\
$1222\bar{1}$ & $00010,\,00100$ & $\mathrm S\ 0000\bar{1}$\\
$12220$ & $00010,\,00100,\,01000$ & $\mathrm P$\\
$12221$ & $00010,\,00100$ & $\mathrm S\ 00001$\\
$12222$ & $00001,\,00010$ & $\mathrm S\ 00100$\\
$2\bar{2}\bar{2}\bar{2}\bar{2}$ & $0\bar{1}000,\,00\bar{1}00,\,000\bar{1}0$ & $\mathrm S\ 0000\bar{1}$\\
$2\bar{2}\bar{2}\bar{2}\bar{1}$ & $0\bar{1}000,\,00\bar{1}00,\,000\bar{1}0$ & $\mathrm S\ 0000\bar{1}$\\
$2\bar{2}\bar{2}\bar{2}0$ & $0\bar{1}000,\,00\bar{1}00,\,000\bar{1}0$ & $\mathrm S\ 10000$\\
$2\bar{2}\bar{2}\bar{2}1$ & $0\bar{1}000,\,00\bar{1}00,\,000\bar{1}0$ & $\mathrm S\ 00001$\\
$2\bar{2}\bar{2}\bar{2}2$ & $0\bar{1}000,\,00\bar{1}00,\,000\bar{1}0$ & $\mathrm S\ 00001$\\
$2\bar{2}\bar{2}2\bar{2}$ & $0\bar{1}000,\,00\bar{1}00$ & $\mathrm S\ 0000\bar{1}$\\
$2\bar{2}\bar{2}2\bar{1}$ & $0\bar{1}000,\,00\bar{1}00$ & $\mathrm S\ 0000\bar{1}$\\
$2\bar{2}\bar{2}20$ & $0\bar{1}000,\,00\bar{1}00,\,00010$ & $\mathrm S\ 10000$\\
$2\bar{2}\bar{2}21$ & $0\bar{1}000,\,00\bar{1}00$ & $\mathrm S\ 00001$\\
$2\bar{2}\bar{2}22$ & $0\bar{1}000,\,00\bar{1}00$ & $\mathrm S\ 00001$\\
$2\bar{2}00\bar{2}$ & $0\bar{1}000,\,0000\bar{1}$ & $\mathrm S\ 10000$\\
$2\bar{2}00\bar{1}$ & $0\bar{1}000,\,10000$ & $\mathrm C\ 00002|0|00110$\\
$2\bar{2}000$ & $0\bar{1}000,\,10000$ & $\mathrm C\ 00011|1|00101$\\
$2\bar{2}001$ & $0\bar{1}000,\,10000$ & $\mathrm C\ 00002|0|00110$\\
$2\bar{2}002$ & $0\bar{1}000,\,00001$ & $\mathrm S\ 10000$\\
$2\bar{2}2\bar{2}\bar{2}$ & $0\bar{1}000,\,000\bar{1}0$ & $\mathrm S\ 0000\bar{1}$\\
$2\bar{2}2\bar{2}\bar{1}$ & $0\bar{1}000,\,000\bar{1}0$ & $\mathrm S\ 0000\bar{1}$\\
$2\bar{2}2\bar{2}0$ & $0\bar{1}000,\,000\bar{1}0,\,00100$ & $\mathrm S\ 10000$\\
$2\bar{2}2\bar{2}1$ & $0\bar{1}000,\,000\bar{1}0$ & $\mathrm S\ 00001$\\
$2\bar{2}2\bar{2}2$ & $0\bar{1}000,\,000\bar{1}0$ & $\mathrm S\ 00001$\\
$2\bar{2}22\bar{2}$ & $0\bar{1}000,\,0000\bar{1}$ & $\mathrm S\ 00010$\\
$2\bar{2}22\bar{1}$ & $0\bar{1}000,\,00010$ & $\mathrm S\ 0000\bar{1}$\\
$2\bar{2}220$ & $0\bar{1}000,\,00010,\,00100$ & $\mathrm S\ 10000$\\
$2\bar{2}221$ & $0\bar{1}000,\,00010$ & $\mathrm S\ 00001$\\
$2\bar{2}222$ & $0\bar{1}000,\,00001$ & $\mathrm S\ 00010$\\
$20\bar{2}0\bar{2}$ & $00\bar{1}00,\,0000\bar{1}$ & $\mathrm S\ 10000$\\
$20\bar{2}0\bar{1}$ & $00\bar{1}00,\,10000$ & $\mathrm C\ 00002|0|01111$\\
$20\bar{2}00$ & $00\bar{1}00,\,10000$ & $\mathrm C\ 00011|1|00101$\\
$20\bar{2}01$ & $00\bar{1}00,\,10000$ & $\mathrm C\ 00002|0|01111$\\
$20\bar{2}02$ & $00\bar{1}00,\,00001$ & $\mathrm S\ 10000$\\
$200\bar{2}\bar{2}$ & $000\bar{1}0,\,0000\bar{1}$ & $\mathrm S\ 10000$\\
$200\bar{2}\bar{1}$ & $000\bar{1}0,\,10000$ & $\mathrm C\ 00002|2|01101$\\
$200\bar{2}0$ & $000\bar{1}0,\,10000$ & $\mathrm C\ 00101|1|01010$\\
$200\bar{2}1$ & $000\bar{1}0,\,10000$ & $\mathrm C\ 00002|2|01101$\\
$200\bar{2}2$ & $000\bar{1}0,\,00001$ & $\mathrm S\ 10000$\\
$2002\bar{2}$ & $0000\bar{1},\,00010$ & $\mathrm S\ 10000$\\
$2002\bar{1}$ & $00010,\,10000$ & $\mathrm C\ 00002|0|01000$\\
$20020$ & $00010,\,10000$ & $\mathrm C\ 00101|1|01111$\\
$20021$ & $00010,\,10000$ & $\mathrm C\ 00002|0|01000$\\
$20022$ & $00001,\,00010$ & $\mathrm S\ 10000$\\
$2020\bar{2}$ & $0000\bar{1},\,00100$ & $\mathrm S\ 10000$\\
$2020\bar{1}$ & $00100,\,10000$ & $\mathrm C\ 00002|0|11100$\\
$20200$ & $00100,\,10000$ & $\mathrm C\ 00011|1|11011$\\
$20201$ & $00100,\,10000$ & $\mathrm C\ 00002|0|11100$\\
$20202$ & $00001,\,00100$ & $\mathrm S\ 10000$\\
$22\bar{2}\bar{2}\bar{2}$ & $00\bar{1}00,\,000\bar{1}0$ & $\mathrm S\ 0000\bar{1}$\\
$22\bar{2}\bar{2}\bar{1}$ & $00\bar{1}00,\,000\bar{1}0$ & $\mathrm S\ 0000\bar{1}$\\
$22\bar{2}\bar{2}0$ & $00\bar{1}00,\,000\bar{1}0,\,01000$ & $\mathrm S\ 10000$\\
$22\bar{2}\bar{2}1$ & $00\bar{1}00,\,000\bar{1}0$ & $\mathrm S\ 00001$\\
$22\bar{2}\bar{2}2$ & $00\bar{1}00,\,000\bar{1}0$ & $\mathrm S\ 00001$\\
$22\bar{2}2\bar{2}$ & $00\bar{1}00,\,0000\bar{1}$ & $\mathrm S\ 00010$\\
$22\bar{2}2\bar{1}$ & $00\bar{1}00,\,00010$ & $\mathrm S\ 0000\bar{1}$\\
$22\bar{2}20$ & $00\bar{1}00,\,00010,\,01000$ & $\mathrm S\ 10000$\\
$22\bar{2}21$ & $00\bar{1}00,\,00010$ & $\mathrm S\ 00001$\\
$22\bar{2}22$ & $00\bar{1}00,\,00001$ & $\mathrm S\ 00010$\\
$2200\bar{2}$ & $0000\bar{1},\,01000$ & $\mathrm S\ 10000$\\
$2200\bar{1}$ & $01000,\,10000$ & $\mathrm C\ 00002|0|01011$\\
$22000$ & $01000,\,10000$ & $\mathrm C\ 00011|1|00101$\\
$22001$ & $01000,\,10000$ & $\mathrm C\ 00002|0|01011$\\
$22002$ & $00001,\,01000$ & $\mathrm S\ 10000$\\
$222\bar{2}\bar{2}$ & $000\bar{1}0,\,0000\bar{1}$ & $\mathrm S\ 00100$\\
$222\bar{2}\bar{1}$ & $000\bar{1}0,\,00100$ & $\mathrm S\ 0000\bar{1}$\\
$222\bar{2}0$ & $000\bar{1}0,\,00100,\,01000$ & $\mathrm S\ 10000$\\
$222\bar{2}1$ & $000\bar{1}0,\,00100$ & $\mathrm S\ 00001$\\
$222\bar{2}2$ & $000\bar{1}0,\,00001$ & $\mathrm S\ 00100$\\
$2222\bar{2}$ & $0000\bar{1},\,00010$ & $\mathrm S\ 00100$\\
$2222\bar{1}$ & $00010,\,00100$ & $\mathrm S\ 0000\bar{1}$\\
$22220$ & $00010,\,00100,\,01000$ & $\mathrm S\ 10000$\\
$22221$ & $00010,\,00100$ & $\mathrm S\ 00001$\\
$22222$ & $00001,\,00010$ & $\mathrm S\ 00100$\\
\end{tabular}\par}

\caption{Certificates for the $210$ vectors $y\in\mathcal Y$ (Theorem~\ref{thm:rank5}).}\label{tab:cert}
\end{table}

\section*{Use of AI tools}
The research for this paper was carried out in interaction with the AI systems Claude (Anthropic, model Opus 5.5)
and, to a lesser extent, ChatGPT (OpenAI). They were used to discuss approaches, to write and run computer searches,
and to draft the manuscript. Drafts were checked by ChatGPT and by the author.

\end{document}